\documentclass[12pt]{amsart}
\usepackage{amsmath}
\usepackage{amssymb}
\usepackage{amscd}
\usepackage{amsthm}

\usepackage[all]{xy}

\newtheorem{thm}{Theorem}[section]
\newtheorem{lem}[thm]{Lemma}
\newtheorem{cor}[thm]{Corollary}
\newtheorem{conj}[thm]{Conjecture}
\newtheorem{prop}[thm]{Proposition}
\newtheorem{question}[thm]{Question}

\theoremstyle{remark}

\theoremstyle{definition}
\newtheorem{defn}[thm]{Definition}

\numberwithin{equation}{section}

\allowdisplaybreaks[4]

\DeclareMathOperator{\rank}{rank}

\DeclareMathOperator{\Pic}{Pic}

\DeclareMathOperator{\pr}{Pr}

\DeclareMathOperator{\CH}{CH}
\DeclareMathOperator{\DIV}{div}

\DeclareMathOperator{\IM}{Im}
\DeclareMathOperator{\kod}{kod}

\DeclareMathOperator{\tr}{tr}

\DeclareMathOperator{\Alt}{Alt}
\DeclareMathOperator{\sq}{i}
\DeclareMathOperator{\cd}{cd}

\begin{document}

\vfuzz0.5pc
\hfuzz0.5pc 

\newcommand{\claimref}[1]{Claim \ref{#1}}
\newcommand{\thmref}[1]{Theorem \ref{#1}}
\newcommand{\propref}[1]{Proposition \ref{#1}}
\newcommand{\lemref}[1]{Lemma \ref{#1}}
\newcommand{\coref}[1]{Corollary \ref{#1}}
\newcommand{\remref}[1]{Remark \ref{#1}}
\newcommand{\conjref}[1]{Conjecture \ref{#1}}
\newcommand{\questionref}[1]{Question \ref{#1}}
\newcommand{\defnref}[1]{Definition \ref{#1}}
\newcommand{\secref}[1]{Sec. \ref{#1}}
\newcommand{\ssecref}[1]{\ref{#1}}
\newcommand{\sssecref}[1]{\ref{#1}}

\def \red{{\mathrm{red}}}
\def \tors{{\mathrm{tors}}}
\def \EQ{\Leftrightarrow}

\def \mapright#1{\smash{\mathop{\longrightarrow}\limits^{#1}}}
\def \mapleft#1{\smash{\mathop{\longleftarrow}\limits^{#1}}}
\def \mapdown#1{\Big\downarrow\rlap{$\vcenter{\hbox{$\scriptstyle#1$}}$}}
\def \smapdown#1{\downarrow\rlap{$\vcenter{\hbox{$\scriptstyle#1$}}$}}
\def \A{{\mathbb A}}
\def \I{{\mathcal I}}
\def \J{{\mathcal J}}
\def \CO{{\mathcal O}}
\def \C{{\mathcal C}}
\def \BC{{\mathbb C}}
\def \BQ{{\mathbb Q}}
\def \M{{\mathcal M}}
\def \H{{\mathcal H}}
\def \Ll{{\mathcal L}}
\def \S{{\mathcal S}}
\def \Z{{\mathcal Z}}
\def \BZ{{\mathbb Z}}
\def \W{{\mathcal W}}
\def \Y{{\mathcal Y}}
\def \T{{\mathcal T}}
\def \P{{\mathbb P}}
\def \CP{{\mathcal P}}
\def \G{{\mathbb G}}
\def \F{{\mathbb F}}
\def \BR{{\mathbb R}}
\def \D{{\mathcal D}}
\def \L{{\mathcal L}}
\def \f{{\mathcal F}}
\def \E{{\mathcal E}}
\def \BN{{\mathbb N}}
\def \N{{\mathcal N}}
\def \X{{\mathcal X}}
\def \CA{{\mathcal A}}
\def \rat{{\text{rat}}}
\def \dgt{{\text{dgt}}}

\def \closure#1{\overline{#1}}
\def \EQ{\Leftrightarrow}
\def \imply{\Rightarrow}
\def \isom{\cong}
\def \embed{\hookrightarrow}
\def \tensor{\mathop{\otimes}}
\def \wt#1{{\widetilde{#1}}}

\title[Real Regulators on Self-products of $K3$ Surfaces]{Real Regulators on Self-products of $K3$ Surfaces}

\author{Xi Chen}

\address{632 Central Academic Building\\
University of Alberta\\
Edmonton, Alberta T6G 2G1, CANADA}

\email{xichen@math.ualberta.ca}   

\author{James D. Lewis}

\address{632 Central Academic Building\\
University of Alberta\\
Edmonton, Alberta T6G 2G1, CANADA}

\email{lewisjd@ualberta.ca}

\date{\today}

\thanks{Both authors partially supported by a grant from the Natural Sciences
and Engineering Research Council of Canada.}



\renewcommand{\abstractname}{Abstract} 
\begin{abstract}  
Based on a novel application 
of an archimedean type pairing to the geometry and
deformation theory of $K3$ surfaces, we construct a regulator indecomposable
$K_1$-class on a self-product of a $K3$ surface. In the Appendix,
we explain how this 
pairing is a special instance of a general  pairing
on precycles in the equivalence relation defining Bloch's
higher Chow groups.
\end{abstract}

\maketitle

\section{Introduction}

Let $X$ be a smooth projective surface. The real regulator map
\begin{equation}\label{E200}
r_{2,1}: \CH^2(X, 1)\to H^{1,1}(X, \BR),
\end{equation}
where $\CH^{\bullet}(-,\bullet)$ are the higher Chow groups defined
in \cite{B},  has been extensively studied. 
The image $\IM(r_{2,1})\tensor \BR$ seems
to behave according to the Kodaira dimension $\kod(X)$ of $X$. Our
knowledge on the subject suggests the following:
\begin{enumerate}
\item if $\kod(X)\le 0$, $\IM(r_{2,1})\tensor \BR = H^{1,1}(X, \BR)$;
this is trivial when $\kod(X) < 0$; for $\kod(X) = 0$, 
a proof is given in \cite{C-L2} for $K3$ and Abelian surfaces;
\item if $\kod(X) > 0$, we expect that
\begin{equation}\label{E201}
\IM(r_{2,1})\tensor \BR\cap H_{\text{tr}}^{1,1}(X, \BR) = \{0\}
\end{equation}
for $X$ general, where $H_\text{tr}^{1,1}(X, \BR)\subset H^{1,1}(X,
\BR)$ is the space of transcendental classes ($=$ orthogonal
complement of algebraic classes); this is known in some special cases such as
surfaces in $\P^3$ and products of curves \cite{C-L1}.
\end{enumerate}

In this paper, we give some evidence that the real regulator
\begin{equation}\label{E001}
r_{3,1}: \CH^3(X\times X, 1) \to H^4(X\times X, \BR)
\end{equation}
on the self product of a smooth projective surface $X$ exhibits a
similar pattern of behavior. In particular, we have the following result for
a general polarized $K3$ surface.

\begin{thm}\label{THM001}
For a general polarized $K3$ surface $(X, L)$, 
\begin{equation}\label{E334}
\IM(\underline{r}_{3,1})\tensor \BR \ne 0
\end{equation}
where $\underline{r}_{3,1}$ is the reduced real regulator
\begin{equation}\label{E335}
\underline{r}_{3,1}: \CH^3(X\times X, 1) \xrightarrow{r_{3,1}} H^4(X\times X, \BR)
\xrightarrow{\text{(projection)}} V(X)
\end{equation}
and $V(X)$ is the subspace of $H^{1,1}(X,\BR) \tensor H^{1,1}(X,\BR)$ given by
\begin{equation}\label{E003}
\begin{split}
V(X) = \{ &\omega\in H^{1,1}(X,\BR) \tensor H^{1,1}(X,\BR):\
\text{for all } \gamma\in H^{1,1}(X,\BR),\\
&\quad \omega\wedge (c_1(L)\tensor
\gamma) = \omega\wedge (\gamma\tensor c_1(L)) = \omega\wedge [\Delta_X] = 0\}.
\end{split}
\end{equation}
Here $\Delta_X$ is the diagonal of $X\times X$ and $[Y]$ is the
Poincar\'e dual of $Y$.
\end{thm}

One should compare the above theorem with the results in \cite{C-L3}, where it
shows that the same regulator map is trivial on a very general product of
$K3$ surfaces.
The proof of the above theorem is similar to that of the Hodge-$\D$-conjecture for $K3$ surfaces in \cite{C-L2}.
We choose a suitable one-parameter family $W$ of $K3$ surfaces over the unit disk $\Delta = \{|t| <
1\}$ and a family of
higher Chow cycles $\xi_t\in \CH^3(W_t\times W_t, 1)$. By studying the limit $\xi_0$ of $\xi_t$ as $t\to 0$ and the
corresponding limit of the regulator maps $r_{3,1}(\xi_t)$, we are able to deduce \eqref{E334}. However, unlike in 
\cite{C-L2}, where the limit $\xi_0$ is a higher Chow cycle itself, $\xi_0$ here is no longer a cycle in
$\CH^3(W_0\times W_0, 1)$. Needless to say, this makes the computation of $\lim_{t\to 0} r_{3,1}(\xi_t)$ much harder.

In the course of our proof, we discover a natural pairing on $z_\rat^*$, which is interesting
in its own right. This pairing, though
quite easy to define and similar to an archimedean height pairing,
to the best of our knowledge, has not been exploited in this situation.  
An abridged version of this pairing is given in \ref{SS002}. In
the Appendix (\secref{SEC003}), we show how this pairing is
a special instance of a generalized pairing on higher cycles
in the equivalence relation defining Bloch's higher Chow groups (\cite{B}).

We wish to point out that the existence of rational curves
on $K3$ surfaces is pivotal to our construction of higher
Chow cycles. On the contrary we anticipate the following:

\begin{conj}\label{CA3}
Let $X = X/\BC\subset \P^3$ be a very general surface of degree $d\geq 5$.
Then the reduced regulator map
$$
\underline{r}_{3,1} : \CH^3(X\times X,1) \to H_{\tr}^4(X\times X,\BR)
\bigcap H^{2,2}(X\times X),
$$
is zero, where $H_{\tr}^4(X\times X,\BR)$ is the space of transcendental cocycles.
\end{conj}

The terminology ``general'' and``very general'' means the following. If $W$ is a
parameter space of a universal family of projective algebraic manifolds of a given class,
then general refers to a point in a nonempty real analytic Zariski open subset of $W$
governed by certain generic properties,
whereas very general refers to a point in the countable intersection of nonempty
Zariski open subsets.

If one works with the notion of indecomposables as for example defined
in \cite{L2}, viz., in our case 
\[
\CH^r_{\text{\rm ind}}(X,1)  :=\ \text{\rm Coker}\big(\CH^1(X,1)\otimes \CH^{r-1}(X)
\ {\buildrel\bigcap\over\longrightarrow}\ \CH^r(X,1)\big),
\]
then \thmref{THM001}  implies the following:

\begin{cor} For a very general polarized $K3$ surface $(X,L)$,
\[
\CH^3_{\text{\rm ind}}(X\times X,1)\otimes\BQ \ne 0.
\]
\end{cor}
 
\section{Real Regulators on Self-products of $K3$ Surfaces}\label{SEC002}

\subsection{Interesting higher Chow cycles}

The definition of higher Chow groups is given in \cite{B} (also see \cite{El-V}).
For the readers convenience, the definition  is also  included in the Appendix.  Let
$C$ and $D$ be two rational curves on $X$, 
$\Delta_C = (C\times C)\cap \Delta_X$ and $\Delta_D = (D\times D)\cap \Delta_X$. 
We assume that $C\in |nL|$ and $D\in |mL|$ and fix a point $p\in C\cap
D$. Let
\begin{equation}\label{E002}
\xi = (f_C, C\times C) + (f_D, D\times D) + (f_\Delta, \Delta_X) + \eta \in \CH^3(X\times
X, 1)
\end{equation}
where 
\begin{equation}\label{E103}
(f_C) = m\Delta_C - m (p\times C) - m (C\times p)
\end{equation}
\begin{equation}\label{E104} 
(f_D) = n (p\times D) + n (D\times p) - n \Delta_D
\end{equation}
and
\begin{equation}\label{E105}
(f_\Delta) = n \Delta_D - m\Delta_C.
\end{equation}
It is trivial to find $\eta$ such that $\DIV(\xi) = 0$. 
More specifically, since $m C\sim_{\text{rat}} n D$, we can find $g\in
\BC(X)^*$ such that $(g) = n D - m C$, where $\sim_{\text{rat}}$ is the
rational equivalence relation. Then we simply let
\begin{equation}\label{E400}
\eta = (\pi_1^* g, X\times p) + (\pi_2^* g, p\times X)
\end{equation}
where $\pi_1$ and $\pi_2$ are the projections from $X\times X$ 
onto the first and second factors respectively. 

\subsection{Deformation of $\xi$}

Next, let us study the deformation of $\xi$ and the corresponding
regulator map as $X$ deforms in the moduli space of polarized $K3$'s.

Let $(X, L)$ be a special polarized $K3$ surface and $M\cup N$ be a union of two rational curves
on $X$ with the following properties:
\begin{enumerate}
\item $\rank_\BZ\Pic(X) = 3$;
\item $M$ and $N$ are nodal;
\item $M$ and $N$ meet transversely at $l = M\cdot N\ge 2$ distinct points;
\item $M$ and $N$ are linearly independent in $H^{1,1}(X, \BQ)$;
\item $M + N\sim_{\text{rat}} kL$ for some integer $k$.
\end{enumerate}
The existence of such $(X, M, N)$ will be proved later.

Now let us consider a general deformation of $X$, i.e., a family of
polarized $K3$ surfaces $W$ over the disk $\Gamma\isom \{|t| < 1\}$ with
$W_0 = X$.
Using the argument in \cite{C1}, we see that $M\cup N$ can be deformed
to a rational curve on a general fiber $W_t$ by smoothing out all but
one intersection between $M$ and $N$.
By that we mean there exists a family of rational curves $\C\subset W$ over $\Delta$, after a suitable base change, such that
$\C_0 = M\cup N$ and $\C_0^\nu = M^\nu \cup N^\nu$ after we normalize $\C$ by $\nu: \C^\nu\to \C$; the two components
$M^\nu\isom\P^1$ and $N^\nu \isom\P^1$ meet transversely at a point $r^\nu$ and $r = \nu(r^\nu)$
is one of the $l$ intersections
in $M\cap N$. So every point in $M\cap N$ except $r$ is ``smoothed'' by $\nu$. For each choice of $r$, we have a
corresponding family $\C$ with the above properties.
So there are exactly $l$
distinct families of rational curves over $\Gamma$ in $W$ with central fiber $M\cup
N$. Let $\C$ and $\D\subset W$ be two of them.

Let $P\subset W$ be a section of $W/\Gamma$ with $P\subset \C\cap
\D$. We will show that $P$ can be chosen such that $p = P_0 \in M\backslash N$.
Now let us consider a higher Chow precycle $\xi$ on $Z = W\times_\Gamma W$,
whose restriction to a general fiber $W_t\times W_t$ is a class
given in \eqref{E002}. 
Here we have $m=n$ with $m$ and $n$ in \eqref{E103}-\eqref{E105}.
That is, we construct a family version of $\xi$ just as above with
$(C, D, p)$ replaced by $(\C, \D, P)$:
\begin{equation}\label{E004}
\xi = (f_\C, \C\times_\Gamma \C) + (f_\D, \D\times_\Gamma \D) + (f_\Delta, \Delta_W) + \eta
\end{equation}
Note that we do not necessarily have $\DIV(\xi) = 0$ although
$\DIV(\xi_t) = 0$ for all $t\ne 0$; $\DIV(\xi)$ might be supported on
some surfaces on the central fiber $Z_0$.
It is easy to see that
\begin{equation}\label{E005}
(f_\C) \ne  \Delta_\C - (P\times_\Gamma \C) -  (\C\times_\Gamma P)
\end{equation}
since otherwise the restriction of $(f_\C)$ to $N\times N$ will be $\Delta_N$, which is
not rationally equivalent to 0. 
Hence, unlike the construction on a single $K3$ surface,
$f_\C$ has zeros or poles outside of $\Delta_\C$, $P\times_\Gamma\C$ and $\C\times_\Gamma P$;
$(f_\C)$ is also supported on some surfaces contained in $W_0\times W_0 = X\times X$.
We claim that $f_\C$ has zeros or poles along $N\times
N$. 
That is, we have
\begin{equation}\label{E109}
(f_\C) =  \Delta_\C - (P\times_\Gamma \C) -  (\C\times_\Gamma P) +
  \mu (N\times N)
\end{equation}
for some $\mu\in\BZ$. Furthermore we claim that $\mu = 1$, viz.,
\begin{equation}\label{E006}
(f_\C) =  \Delta_\C - (P\times_\Gamma \C) -  (\C\times_\Gamma P) + (N\times N).
\end{equation}

Note that $f_\C$ is a rational function on $\C\times_\Gamma \C$. Or better, we should think of 
its pullback $f_\C^\nu = \nu^* f_\C$ on $\C^\nu \times_\Gamma C^\nu$, where
$\nu: \C^\nu\to \C$ is the normalization of $\C$. 
As described before, the general
fiber of $\C^\nu$ is a smooth rational curve and hence the central
fiber $\C_0^\nu$ has to be the union $M^\nu\cup N^\nu$,
where $M^\nu$ and $N^\nu$ are the normalizations of $M$ and $N$,
respectively, and $M^\nu$ and $N^\nu$ meet transversely at a
point $r^\nu$ lying over $r = \nu(r^\nu)\in M\cap N$. To reiterate, the
$l$ different choices of $r\in M\cap N$ give arise to $l$ distinct
families of rational curves in $W$ with central fiber $M\cup N$.
Now on $\C^\nu\times_\Gamma \C^\nu$,
\begin{equation}\label{E950}
\Delta_\C \not\sim_\rat (P^\nu\times_\Gamma \C^\nu) + (\C^\nu\times_\Gamma P^\nu)
\end{equation}
as pointed out in \eqref{E109}, where
$P^\nu =
\nu^{-1}(P)$. 
On the other hand, the two sides of \eqref{E950} are rationally equivalent when restricted to the
general fibers. Therefore, we have
\begin{equation}\label{E951}
\begin{split}
\Delta_C \sim_\rat & (P^\nu\times_\Gamma \C^\nu) + (\C^\nu\times_\Gamma P^\nu) - \mu(N^\nu \times N^\nu)\\
&\quad - \mu_M (M^\nu\times N^\nu) - \mu_N (N^\nu\times M^\nu).
\end{split}
\end{equation}
When we restrict \eqref{E951} to $M^\nu\times M^\nu$, we have
\begin{equation}\label{E952}
\Delta_{M^\nu}\sim_\rat
(P_0^\nu \times M^\nu) + (M^\nu\times P_0^\nu) - \mu_M(M^\nu\times r^\nu) - \mu_N(r^\nu\times M^\nu).
\end{equation}
This forces that $\mu_M = \mu_N = 0$.
Similarly,
when we restrict it to $M^\nu \times N^\nu$, we have
\begin{equation}\label{E110}
-\mu (r^\nu \times N^\nu) + (P_0^\nu \times N^\nu) \sim_\rat 0,
\end{equation}
and hence $\mu = 1$.

One thing worth noting is that the threefold $\C^\nu\times_\Gamma \C^\nu$ is actually singular. It has a rational
double point at $r^\nu\times r^\nu$. The above argument works nevertheless. One may choose to work with a
desingularization of $\C^\nu\times_\Gamma \C^\nu$; the same argument works almost without any change.
 
Similarly, we have
\begin{equation}\label{E007}
(f_\D) = (P\times_\Gamma \D) + (\D\times_\Gamma P) - \Delta_\D - (N\times N).
\end{equation}
We see that the terms $(N\times N)$ cancel each other out. Hence $\xi$ extends to a
higher Chow cycle on $Z$. Note that $Z$ is an analytic space; so a higher Chow
cycle in this context is likewise regarded as an analytic cycle.

Next, we will attend to the calculation of $r_{3,1}(\xi)(\omega)$ for an algebraic
class $\omega\in V(X)$ on the central fiber.
Our objective is to show that
\begin{equation}\label{E009}
r_{3,1}(\xi)(i_* \omega) = r_{3,1}(i^*\xi) (\omega) \ne 0
\end{equation}
for some $\omega\in V(X)$, where $i$ is the inclusion
 $Z_0\hookrightarrow Z$ and $i_*$ is the Gysin map
\begin{equation}\label{E210}
i_*: H^{2,2}(Z_0) \to H^{3,3}(Z),
\end{equation}
defined on the level of currents.
Note that $H^{3,3}(Z)$ is now Dolbeault (again, as $Z$ is analytic).
Actually, we will restrict $\omega$ to a subspace of $V(X)$. Let
\begin{equation}\label{E111}
\hat V(X) = V(X)\cap [M\times M]^\perp \cap [N\times N]^\perp \cap
H^{4}_{\text{\rm alg}}(X\times X,\BQ),
\end{equation}
where the latter term is the space of algebraic cocycles. 
Note that the condition $\rank_\BZ \Pic(X) = 3$ implies that
$\hat V(X)$ is nontrivial.
We want to show that
\begin{equation}\label{E112}
r_{3,1}(\xi)(i_* \omega) = r_{3,1}(i^*\xi) (\omega) \ne 0
\end{equation}
for $\omega\ne 0\in \hat V(X)$.
Our main technical difficulty to compute $r_{3,1}(\xi)(i_* \omega)$
lies in the fact that $f_\C$ vanishes on $N\times N$.
To overcome this, we blow up $Z$ along $N\times N$. Let $\pi: \wt{Z}\to Z$ be the
blowup with exceptional divisor $E$. Then from \eqref{E004},
\begin{equation}\label{E011}
\pi^* \xi = (\wt{f}_\C, \wt{\C\times_\Gamma \C}) + 
(\wt{f}_\D, \wt{\D \times_\Gamma \D}) + (\wt{f}_\Delta, \wt{\Delta}_W)
+ \pi^* \eta + \alpha
\end{equation}
where $\wt{\C\times_\Gamma \C}$, $\wt{\D\times_\Gamma \D}$
and $\wt{\Delta}_W\subset \wt{Z}$ are the proper transforms of $\C\times_\Gamma \C,
\D\times_\Gamma \D$ and
$\Delta_W$ under $\pi$, respectively, and $\alpha$ is a higher Chow precycle
supported on $E$. It is not hard to see that
\begin{equation}\label{E012}
\DIV(\alpha) = - (\wt{\C\times_\Gamma \C}\cap E) + (\wt{\D\times_\Gamma \D} \cap E).
\end{equation}
Both $\wt{\C\times_\Gamma \C}\cap E$ and $\wt{\D\times_\Gamma \D} \cap E$ are
rational sections of the $\P^2$ bundle $E$ over $N\times N$.
Note that \eqref{E011} and \eqref{E012} does not determine $\pi^*\xi$
uniquely. But as far as the value of
\begin{equation}\label{E013}
r_{3,1}(\xi)(i_* \omega) = r_{3,1}(\pi^* \xi)(\pi^* i_* \omega)
\end{equation}
is concerned, it does not really matter since if we choose $\alpha$
differently, say $\alpha'$, then (using \cite{B}):
\begin{equation}\label{E350}
\alpha' - \alpha
\in (i_E)_* \CH^2(E,1) \simeq \CH^2(N\times
N,1) \oplus \CH^1(N\times N,1),
\end{equation}
where $i_E$ is the embedding $E\hookrightarrow \wt{Z}$. If we assume for the moment that $N$ is smooth, then
\begin{equation}\label{E351}
r_{3,1}(\alpha' - \alpha)(\pi^* i_* \omega) = 0
\end{equation}
for $\omega\in \hat V(X)$ since $\omega \wedge [N \times N]
= 0$ for such $\omega$. In summary, the value of $r_{3,1}(\pi^*
\xi)(\pi^* i_* \omega)$ is independent of the choice of
$\alpha$, once one first passes to a normalization of $N$.
This is the crucial observation which enables us to compute
the regulators in \eqref{E013} without knowing the exact form of $\pi^*\xi$.

Let $\wt{Z}_0$ be the proper transform of $Z_0 = W_0\times W_0$ under
the blowup $\pi: \wt{Z}\to Z$. We have the commutative diagram
\begin{equation}\label{E211}
\xymatrix{\wt{Z}_0 \ar[r]^{\wt{i}} \ar[d]^{\pi_0} & \wt{Z} \ar[d]^\pi\\
Z_0 \ar[r]^i & Z
}
\end{equation}
where $\wt{i}$ is the inclusion $\wt{Z}_0\hookrightarrow \wt{Z}$.
It is not hard to see that
\begin{equation}\label{E212}
\pi^* i_* \omega = \wt{i}_* \pi_0^* \omega,
\end{equation}
using the fact that $\omega\in \hat{V(X)} \subset  [N \times N]^{\perp}$.
Therefore, \eqref{E013} becomes
\begin{equation}\label{E213}
r_{3,1}(\pi^* \xi)(\pi^* i_* \omega) = 
r_{3,1}(\pi^* \xi)(\wt{i}_* \pi_0^* \omega)
= r_{3,1}(\wt{i}^* \pi^* \xi)(\pi_0^* \omega).
\end{equation}

When we restrict $\pi^* \xi$  (see \eqref{E011})  to $\wt{Z}_0$, we have
\begin{equation}\label{E113}
\begin{split}
\wt{i}^* \pi^*\xi
&= (c_{M\times M}, \wt{M\times M}) + (c_{\Delta}, \wt{\Delta}_W) +
\wt{i}^* \pi^* \eta\\
&\quad +
(\phi_{M\times N}, \wt{M\times N}) + 
(\phi_{N\times M}, \wt{N\times M}) + 
\wt{i}^* \alpha,
\end{split}
\end{equation}
where $c_\bullet$ are nonzero constants, $\wt{M\times M}, \wt{M\times
  N}$ and $\wt{N\times M}$ are proper transforms of $M\times M$,
  $M\times N$ and $N\times M$, 
and $\phi_{M\times N}$ and
$\phi_{N\times M}$ are rational functions on $\wt{M\times N}$ and
$\wt{N\times M}$, respectively. Note that $\wt{M\times N}$ and
  $\wt{N\times M}$ are blowups of $M\times N$ and $N\times M$ along
  $N\times N$ and hence
\begin{equation}\label{E301}
\wt{M\times N}\isom M\times N \text{ and }
\wt{N\times M}\isom N\times M,
\end{equation}
using the assumption that $M$ and $N$ meet transversally.
So we may write $M\times N$ and $N\times M$ for $\wt{M\times N}$ and
$\wt{N\times M}$ wherever no confusion is possible.
It is easy to see that 
\begin{equation}\label{E300}
r_{3,1}(\wt{i}^* \pi^* \eta)(\pi_0^*\omega) = 0.
\end{equation}
Also the contributions of the first two terms on the RHS of \eqref{E113}
are zero as well since
\begin{equation}\label{E302}
\omega\wedge [M\times M] = \omega \wedge [\Delta_W] = 0.
\end{equation}
So all the nontrivial contributions to
$r_{3,1}(\wt{i}^* \pi^* \xi)(\pi_0^* \omega)$
come from the last three terms of \eqref{E113}.
It follows from \eqref{E110} that
\begin{equation}\label{E303}
(\phi_{M\times N}) = (r_1\times N) - (r_2\times N),
\end{equation}
where $r_1$ and $r_2$ are two distinct points among the intersection $M\cap N$.
Similarly,
\begin{equation}\label{E304}
(\phi_{N\times M}) = (N\times r_1) - (N\times r_2).
\end{equation}
More precisely, there exists a rational function $\phi\in \BC(M)^*$
such that
\begin{equation}\label{E305}
(\phi) = r_1 - r_2, \ \phi_{M\times N} = (\rho_1)^* \phi, \
  \text{and}\
\phi_{N\times M} = (\rho_2)^* \phi
\end{equation}
where $\rho_1$ and $\rho_2$ are the projections
$M\times N\to M$ and $N\times M\to M$, respectively.
It remains to figure out how to compute the term
\begin{equation}\label{E114}
r_{3,1}(\wt{i}^* \alpha)(\pi_0^*\omega).
\end{equation}
Once that is done, we will know exactly how to compute
$r_{3,1}(\wt{i}^* \pi^* \xi)(\pi_0^* \omega)$. Then by an
appropriate choice of $\omega$, we will arrive at \eqref{E112}.

Obviously, $\pi_0:\wt{Z}_0\to Z_0$ is the blowup of $Z_0$ along
$N\times N$ with exception divisor $E_0 = E\cap \wt{Z}_0$. 
Note that $E_0$ is a $\P^1$ bundle over $N\times N$.
Since
$\wt{i}^* \alpha$ is supported on $E_0$, 
it can be regarded as a higher Chow precycle on
$E_0$.
More precisely, let $\nu: E_0^\nu\to \wt{Z}_0$ be the normalization of
$E_0$. There exists a higher Chow precycle
\begin{equation}\label{E115}
\beta \in C_{\text{pre}}^2(E_0^\nu, 1)
\end{equation}
such that $\nu_* \beta = \wt{i}^* \alpha$. Then
\begin{equation}\label{E116}
r_{3,1}(\wt{i}^* \alpha)(\pi_0^*\omega) 
= r_{2,1}(\beta)(\nu^* \pi_0^* \omega)
\end{equation}
where
\begin{equation}\label{E207}
\DIV(\beta) = - \nu^*\wt{i}^* (\wt{\C\times_\Gamma \C}) + \nu^*\wt{i}^* (\wt{\D\times_\Gamma \D}).
\end{equation}

\subsection{Pairing on $z_{\text{rat}}^*$}\label{SS002}

The RHS of \eqref{E116} can be put in a more general context
as follows.
Let $Y$ be a smooth projective variety of pure
dimension $n$. Let $z_{\text{rat}}^k(Y)$ be the subgroup of the algebraic
cycles $z^k(Y)$ of codimension $k$ on $Y$, that are rationally equivalent to
zero. Note that for every $\eta\in z_{\text{rat}}^k (Y)$, there exists
$\beta\in C_{\text{pre}}^{k}(Y, 1)$ with 
\begin{equation}\label{E204}
\DIV(\beta) = \eta.
\end{equation}
We have a pairing
\begin{equation}\label{E202}
\langle\ ,\  \rangle: z_{\text{rat}}^k(Y) \times
z_{\text{rat}}^{n+1-k}(Y)\to \BR
\end{equation}
given by
\begin{equation}\label{E203}
\langle \eta, \varepsilon\rangle = r_{k,1}(\beta)(\varepsilon) = \sum
\int_{D_i\cap \varepsilon} \log |f_i|
\end{equation}
where $\eta\in z_{\text{rat}}^k(Y)$, $\varepsilon\in
z_{\text{rat}}^{n+1-k}(Y)$ and $\beta = \sum (f_i, D_i)$ is a higher Chow precycle
satisfying \eqref{E204}. Note that the RHS of
\eqref{E203} is only well-defined for $|\eta|\cap |\varepsilon| =
\emptyset$. A less obvious statement however, is that $\beta$
can be chosen so that $D_i\cap \varepsilon$ is a proper intersection for all $i$
(hence is a zero-cycle).   Although we explain  this in detail
in the Appendix, the existence of $\beta$ can be deduced from
\cite{B} (Lemma (4.2)),  together with
a standard norm argument.
It is easy to see that the pairing is well-defined,
i.e., it is independent of the exact choice of $\beta$, since if
$\DIV(\beta - \beta') = 0$, then $\beta - \beta'$ is a higher Chow
cycle and hence
\begin{equation}\label{E205}
r_{k,1}(\beta - \beta')(\varepsilon) = 0
\end{equation}
as $\varepsilon \sim_{\text{rat}} 0$. 
It is also obvious that this pairing extends naturally to
$z_{\text{rat}}^*(Y)\tensor_\BZ \BQ$.
The projection formula holds trivially from the definition. That is, we
have
\begin{prop}\label{PROP003}
Let $\pi: X\to Y$ be a flat surjective morphism between two smooth
projective varieties $X$ and $Y$. Then
$\langle \eta, \pi^* \varepsilon\rangle = \langle \pi_* \eta,
\varepsilon\rangle$ for all $\eta\in z_\text{rat}^k(X)$ and
$\varepsilon \in z_\text{rat}^{m - k+1}$ with $|\pi_* \eta|\cap
|\varepsilon| = \emptyset$,
where $m = \dim X$.
\end{prop}

A little less obvious fact is that this pairing is symmetric. That is,
it has the following property which we will call the {\it reciprocity
property\/} of the pairing \eqref{E202}.

\begin{prop}\label{PROP001}
For all $\eta\in z_{\text{rat}}^k(Y)$ and $\varepsilon\in
z_{\text{rat}}^{n+1-k}(Y)$ with $|\eta|\cap |\varepsilon| = \emptyset$,
$\langle \eta, \varepsilon\rangle = \langle \varepsilon, \eta\rangle$.
\end{prop}

\begin{proof} This can be deduced from 
 \thmref{THMA1} on a generalized pairing, in the Appendix;
however it is instructive to give a direct proof of this.
Let $(f, D)$ and $(g, E)$ be the higher Chow precycles such that $\eta =
\DIV(f)$ and $\varepsilon = \DIV(g)$.
Again, by Lemma (4.2) in \cite{B}, we can
assume that
with regard to the pairs $(f,D)$, $(g,E)$, everything is
in ``general'' position. For notational simplicity, let
us assume that $D$ and $E$ are irreducible and meet properly
along an irreducible curve $C$. Let
\begin{equation}\label{E250}
\begin{aligned}
f_c := f\big|_C \in \BC(C)^* &\\
g_c := g\big|_C\in \BC(C)^* &
\end{aligned}
\end{equation}
For every point $p\in C$, put
\begin{equation}\label{E251}
T_p\{f_c,g_c\} = 
(-1)^{\nu_p(f_c)\nu_p(g_c)}\biggl(\frac{f_c^{\nu_p(g_c)}}{g_c^{\nu_p(f_c)}}\biggr)_p.
\end{equation}
where $\nu_p(h)$ is the vanishing order of a function $h$
at $p$.
Since $|\varepsilon| \cap |\eta| = \emptyset$, it follows that
\begin{equation}\label{E252}
T_p\{f_c,g_c\} =
\begin{cases}
f_c^{\nu_p(g_c)}(p)&\text{if $\nu_p(g_c) \ne
  0$}\\
g_c^{-\nu_p(f_c)}(p)&\text{if $\nu_p(f_c) \ne 0$}\\
1&\text{otherwise}
\end{cases}
\end{equation}
Then it is a consequence of Weil reciprocity:
\begin{equation}\label{E253}
\prod_{p\in C}T_p\{f_c,g_c\} = 1.
\end{equation}
that
\begin{equation}\label{E254}
\int_{D\cap \DIV(g)} \log|f| = \int_{E\cap \DIV(f)} \log|g|.
\end{equation}
Obviously, this is equivalent to $\langle \eta, \varepsilon\rangle =
\langle \varepsilon, \eta\rangle$.
\end{proof}

In addition, this pairing is also nondegenerate.

\begin{prop}\label{PROP002}
If $\langle \eta, \varepsilon\rangle = 0$
for all $\varepsilon\in
z_{\text{rat}}^{n+1-k}(Y)$ with $|\eta|\cap |\varepsilon| =
\emptyset$,
then $\eta = 0$.
\end{prop}

\begin{proof}
Let $(f, D)$ and $(g, E)$ be the higher Chow cycles such that
$\eta = \DIV(f)$ and $\varepsilon = \DIV(g)$.
Assume to the contrary that
$\eta  \ne 0$ and choose $E$ such that
\begin{equation}\label{E455}
E \bigcap \eta  = \sum_{i=1}^N(p_i-q_i),\quad
\{p_1,\ldots,p_N\} \bigcap \{q_1,\ldots,q_N\} = \emptyset.
\end{equation}
So it suffices to find $g\in \BC(E)^*$ such
that
\begin{equation}\label{E456}
\log\biggl|\prod_{i=1}^N\frac{g(p_i)}{g(q_i)}\biggr| =
\sum_{i=1}^N\big[\log|g(p_i)| - \log|g(q_i)|\big] \ne 0.
\end{equation}
This is obvious for $E = \P^1$. For arbitrary $E$, it is enough to
take a general projection $E\dashrightarrow \P^1$.
\end{proof}

Now let us go back to \eqref{E116}.
Its RHS can be interpreted as the pairing
\begin{equation}\label{E206}
r_{2,1}(\beta)(\nu^*\pi_0^*\omega)
= \langle \DIV(\beta), \nu^*\pi_0^*\varepsilon \rangle
\end{equation}
where $\DIV(\beta)$ is given in \eqref{E207}
and $\varepsilon\in z^2(Z_0)$ is an algebraic cycle with
$[\varepsilon] = \omega$. It
is easy to see that $\nu^*\pi_0^* \varepsilon\in
z_{\text{rat}}^2(E_0^\nu)$ for $\varepsilon$ with $[\varepsilon] =
\omega$ since
\begin{equation}\label{E209}
\omega \wedge [N\times N] = 0.
\end{equation}
For a general choice of $\varepsilon$, we clearly have $|\DIV(\beta)|\cap
|\nu^* \pi_0^* \varepsilon|=\emptyset$.
Let $N^\nu\isom \P^1$ be the normalization of $N$. Then $E_0^\nu$ is a $\P^1$
bundle over $N^\nu\times N^\nu$. 
We have the commutative diagram
\begin{equation}\label{E214}
\xymatrix{E_0^\nu \ar[r]^{\nu} \ar[d]^{\pi_N} & \wt{Z}_0 \ar[d]^{\pi_0}\\
N^\nu\times N^\nu \ar[r]^(0.6){\nu_N} & Z_0}
\end{equation}
Hence, $\nu^*\pi_0^*\varepsilon = \pi_N^* \nu_N^*\varepsilon$ and
\begin{equation}\label{E215}
\langle \DIV(\beta), \nu^*\pi_0^*\varepsilon \rangle =
\langle \DIV(\beta), \pi_N^* \nu_N^*\varepsilon \rangle = 
\langle (\pi_N)_* \DIV(\beta), \nu_N^*\varepsilon \rangle
\end{equation}
for all $\varepsilon\in z^2(Z_0)$ 
with $[\varepsilon]\in \hat V(X)$ and $|\DIV(\beta)|\cap
|\pi_N^* \nu_N^*\varepsilon|=\emptyset$.
Since $\DIV(\wt{i}^* \pi^*\xi) = 0$, it follows from \eqref{E303} and
\eqref{E304} that
\begin{equation}\label{E306}
(\pi_N)_* \DIV(\beta) = - (r_1^\nu\times N^\nu) - (N^\nu\times
  r_1^\nu) + (r_2^\nu\times N^\nu) + (N^\nu\times r_2^\nu),
\end{equation}
where $r_1^\nu$ and $r_2^\nu$ are the points over $r_1$ and $r_2$, respectively,
under the normalization $N^\nu\to N$. Combining this with
\eqref{E305}, we have
\begin{equation}\label{E307}
\begin{split}
r_{3,1}(\wt{i}^* \pi^* \xi)(\pi_0^* \omega)
&= \int_{\varepsilon\cap M\times N} \log |\phi_{M\times N}|
+ \int_{\varepsilon\cap N\times M} \log |\phi_{N\times M}|\\
&\quad+ \langle (\pi_N)_* \DIV(\beta), \nu_N^*\varepsilon \rangle\\
&= \int_{(\rho_1)_* \varepsilon} \log |\phi|
+ \int_{(\rho_2)_* \varepsilon} \log |\phi|\\
&\quad + \int_{(p_1)_* \varepsilon} \log |\varphi| 
+ \int_{(p_2)_* \varepsilon} \log |\varphi|
\end{split}
\end{equation}
where $p_1$ and $p_2$ are two projections $N\times N\to N$ and
$\varphi\in \BC(N)^*$ is a rational function with $(\varphi) = r_2 - r_1$. 
Here by $(\rho_1)_* \varepsilon$ we really mean
$(\rho_1)_* (\varepsilon\cap M\times N)$.
It remains to find appropriate $\varepsilon$ such that the RHS of
\eqref{E307} is nonzero.

\subsection{Choice of $\varepsilon$}

Consider $\varepsilon = \gamma\tensor \delta$, where $\gamma,\delta\in
z^2(X)$ satisfy
\begin{equation}\label{E308}
[\gamma] \wedge [M] = [\gamma]\wedge [N] = 0
\end{equation}
and
\begin{equation}\label{E311}
[\delta]\wedge [M+N] =
[\gamma]\wedge [\delta] = 0.
\end{equation}
Obviously, $[\varepsilon]\in \hat V(X)$ for such $\gamma$ and
$\delta$. By \eqref{E308},
\begin{equation}\label{E309}
(\rho_2)_* \varepsilon = \deg(\gamma \cdot N) (\delta\cdot M) = 0
\end{equation}
and
\begin{equation}\label{E401}
(p_2)_* \varepsilon = \deg(\gamma \cdot N) (\delta \cdot N) = 0.
\end{equation}
Therefore,
\begin{equation}\label{E310}
\begin{split}
r_{3,1}(\wt{i}^* \pi^* \xi)(\pi_0^* \varepsilon) 
&= \deg(\delta\cdot N)\left(
\int_{\gamma\cap M} \log |\phi| + \int_{\gamma\cap N}
\log |\varphi|\right)\\
&= \deg(\delta\cdot N) \bigg(
\langle r_1 - r_2, \gamma\rangle_M
+ \langle r_2 - r_1, \gamma \rangle_N
\bigg)
\end{split}
\end{equation}
where we use $\langle\ ,\ \rangle_M$ and $\langle\ ,\ \rangle_N$ for the
pairings:
\begin{equation}\label{E314}
\langle\ ,\ \rangle_M: z_\text{rat}^1(M)\times z_\text{rat}^1(M) \to \BR
\end{equation}
and
\begin{equation}\label{E900}
\langle\ ,\ \rangle_N: z_\text{rat}^1(N)\times z_\text{rat}^1(N) \to \BR
\end{equation}
respectively.
Note that for every $\gamma$ satisfying \eqref{E308}, we can always
find $\delta$ satisfying \eqref{E311} and $[\delta]\wedge [N] \ne 0$
by dimension counting. Therefore, $\deg(\delta\cdot N) \ne 0$ and it
suffices to find $\gamma$ satisfying \eqref{E308} and
\begin{equation}\label{E312}
\langle r_1 - r_2, \gamma\rangle_M
+ \langle r_2 - r_1, \gamma \rangle_N \ne 0.
\end{equation}

As a side note, we see that the LHS of \eqref{E312} vanishes for
$\gamma\in z_\text{rat}^1(X)$ and $r_1, r_2\not\in |\gamma|$;
if $\gamma = (g)$ for some $g\in \BC(X)^*$ and $g(r_1)g(r_2)\ne 0$, then
\begin{equation}\label{E313}
\begin{split}
&\quad\langle r_1 - r_2, \gamma\rangle_M + \langle r_2 - r_1, \gamma \rangle_N\\
&= \langle \gamma, r_1 - r_2 \rangle_M + \langle \gamma,
  r_2 - r_1\rangle_N\\
&= (\log |g(r_1)| - \log |g(r_2)|) + (\log |g(r_2)| - \log |g(r_1)|) = 0.
\end{split}
\end{equation}
This is consistent with the fact that $r_{3,1}(\wt{i}^* \pi^* \xi)(\pi_0^*
\varepsilon) = 0$ for $[\varepsilon] = 0$. 
Let 
\begin{equation}\label{E315}
\mu: M^\nu\cup_{\{r_1,r_2\}} N^\nu\to X
\end{equation}
be a partial normalization of $M\cup
N$ that normalizes every singularity of $M\cup N$ except $r_1$ and
$r_2$. That is, $M^\nu\cup_{\{r_1,r_2\}} N^\nu$ is the union of
$M^\nu\isom \P^1$ and $N^\nu\isom \P^1$ meeting transversely at two
points which we still denote by $r_1$ and $r_2$.

Consider $\Pic^{0,0}(M^\nu\cup_{\{r_1,r_2\}} N^\nu)$, which is the Picard
group of Cartier divisors (line bundles) on $M^\nu\cup_{\{r_1,r_2\}}
N^\nu$ whose degrees are zero  when restricted to $M^\nu$ and $N^\nu$,
respectively. We have a well-defined map
\begin{equation}\label{E316}
\Pic^{0,0}(M^\nu\cup_{\{r_1,r_2\}} N^\nu) \xrightarrow{} \BC^*
\end{equation}
sending
\begin{equation}\label{E317}
\gamma \xrightarrow{}
\left.\left(\frac{\phi_{M,\gamma}(r_1)}{\phi_{M,\gamma}(r_2)}\right)
\right/\left(\frac{\phi_{N,\gamma}(r_1)}{\phi_{N,\gamma}(r_2)}\right)
\end{equation}
where $\phi_{M,\gamma}\in \BC(M^\nu)^*$ and $\phi_{N,\gamma}\in \BC(N^\nu)^*$ are rational functions such that
\begin{equation}\label{E318}
\gamma = (\phi_{M,\gamma})  + (\phi_{N,\gamma}).
\end{equation}
It is well known that $\Pic^{0,0}(M^\nu\cup_{\{r_1,r_2\}} N^\nu) \isom \BC^*$.
Actually, \eqref{E316} gives such an isomorphism.
We have the natural pullback map
\begin{equation}\label{E321}
\mu^*: M^\perp \cap N^\perp
\to\Pic^{0,0}(M^\nu\cup_{\{r_1,r_2\}} N^\nu)
\end{equation}
where
\begin{equation}\label{E322}
M^\perp \cap N^\perp
= \{ \gamma\in\Pic(X): \gamma\cdot M = \gamma\cdot N = 0 \}.
\end{equation}
Combining this map with \eqref{E316}, 
we have the map
\begin{equation}\label{E319}
h: M^\perp \cap N^\perp \xrightarrow{\mu^*}                                            
\Pic^{0,0}(M^\nu\cup_{\{r_1,r_2\}} N^\nu)
\xrightarrow{\sim} \BC^*
\end{equation}
for which it is easy to see that
\begin{equation}\label{E320}
\log|h(\gamma)| = \langle r_1 - r_2, \gamma\rangle_M + \langle r_2 - r_1, \gamma \rangle_N.
\end{equation}
Therefore, it comes down to find $\gamma\in M^\perp\cap N^\perp$ such
that $\log|h(\gamma)| \ne 0$. We will show such $\gamma$ exists for a
general deformation of $(X, M, N)$. More precisely, we expect the following
to be true.

\begin{conj}\label{CONJ001}
Let $X$ be a $K3$ surface and $M$ and $N$ are two 
rational curves on $X$
that meets transversely along at least two distinct points $r_1$ and
$r_2$. Suppose that $M^\perp\cap N^\perp\ne 0$. Let $(\X, \M, \N)$ be
a general deformation of $(X, M, N)$. That is, $\X$ is a family of $K3$
surfaces over a quasi-projective curve $\Gamma$ and $\M$ and $\N\subset
\X$ are two families of rational curves over $\Gamma$ with $(\X_0,
\M_0, \N_0) = (X, M, N)$. Let 
\begin{equation}\label{E323}
\xymatrix{\M^\perp\cap \N^\perp \ar[r]^h \ar@{=}[d] & \BC^*\\
\{ \gamma\in \Pic(\X/\Gamma): \gamma\cdot \M = \gamma\cdot \N = 0 \} &
}
\end{equation}
be the family version of the map \eqref{E319}, where we fix two
sections $R_1$ and $R_2$ of $\X/\Gamma$ with $R_i\subset \M\cap \N$
and $R_i\cap \X_0 = r_i$
for $i=1,2$. Then $\BC^*$
is dominated by one of the components of $\M^\perp\cap \N^\perp$ under
the map $h$.
\end{conj}

We state the above as a conjecture since we won't be able to prove it
in full generality. However, we only need to prove it for some special
triples $(X, M, N)$ anyway. So now we come to the geometric side of
the problem, i.e., to find such $(X, M, N)$ with the required
properties at the very beginning of our construction and for which
\conjref{CONJ001} holds.

\subsection{Construction of $(X, M, N)$}

Let $X$ be a $K3$ surface with Picard lattice
\begin{equation}\label{E324}
\begin{bmatrix}
2m & 1 & 1\\
1  & -2 & 2\\
1  & 2 & -2
\end{bmatrix}
\end{equation}
That is, $\Pic(X)$ is generated by $E_1, E_2$ and $G$ with $E_i^2 =
-2$, $E_1\cdot E_2 = 2$, $E_i\cdot G = 1$ and $G^2 = 2m$, where $m = -1$ or $0$.
The pencil $|E_1+E_2|$ realizes $X$ as an elliptic fiberation $X\to
\P^1$. It has exactly $23$ singular fibers; one of them is $E_1\cup
E_2$ with two smooth rational curves $E_1$ and $E_2$ meeting transversely at two
points; the other $22$ singular fibers are rational curves each with
exactly one node. Such $K3$ surface can be polarized by the very ample
divisor $G + k(E_1 + E_2)$ which results in a $K3$ surface of genus
$2k+m+1$. By choosing different combinations of $(m, k)$, $(X,
G+k(E_1+E_2))$ lies on every irreducible component of the moduli space
of polarized $K3$ surfaces. In other word, every polarized $K3$ surfaces
can be deformed to $(X, G+k(E_1 + E_2))$ for some $m$ and $k$.

Such $X$ can be explicitly constructed as a double cover of $S
= \P^1\times \P^1$ or $\F_1$ ramified over a smooth curve $R\in
|-2K_S|$, where $K_S$ is the canonical divisor of $S$. If we take $R$
to be a general member of the linear system, $\Pic(X)$ has only rank
$2$. So $R$ has to be special. It is not hard to
see that $R$ has the following property.

Let $\Pic(S)$ be generated by $C$ and
$F$ with $C^2 = m$, $C\cdot F=1$ and $F^2 = 0$. Then 
there exists a curve $E\in |F|$ such that $E$ is tangent
to $R$ at two points each with multiplicity $2$, i.e., $E$ is a
bitangent of $R$. It
is easy to see that $\pi^{-1} (E) = E_1\cup E_2$ and $X$ has Picard
lattice \eqref{E324} with $G = \pi^* C$, where $\pi$ is the double
coverring map $X\to S$. We can say a lot about $X$ through this
representation of $X$.

We choose $M$ and $N$ to be two rational curves in the linear series
$|E_1 + E_2|$ and $|G + (k-1) (E_1 + E_2)|$, respectively. First, we
need to verify the following.

\begin{prop}\label{PROP004}
Let $X$ be a general $K3$ surface with Picard lattice \eqref{E324}. Then
for every $k\ge 1$, there exists rational curves 
\begin{equation}\label{E326}
M\in |E_1+E_2| \text{ and } N\in |G + (k-1) (E_1 + E_2)|
\end{equation}
such that both $M$ and $N$ are nodal
and they meet transversely at two points.
\end{prop}

\begin{proof}
We basically follow the same idea in \cite{C1}. Every $K3$ surface can
be degenerated to a union of two rational surfaces. In this case, $X$
is not arbitrarily general but we can still degenerate it to a union
of rational surfaces. That is, there exists a family of surfaces $\X$
over a smooth projective curve $\Gamma$ whose general fibers are smooth $K3$
surfaces with Picard lattice \eqref{E324} and whose fiber $\X_0
= S_1 \cup S_2$ over a fixed point $0\in \Gamma$
is a union of two surfaces with $S_1, S_2\isom
\P^1\times \P^1$ if $m = 0$ and $S_1, S_2\isom \F_1$ if $m = -1$.

Just as in \cite{C1}, the two smooth rational surfaces $S_1$ and $S_2$ meet transversely along a smooth
elliptic curve $D\in |-K_{S_i}|$ and the threefold $\X$ has $16$ rational double
points $p_1, p_2, ..., p_{16}$ lying on $D$. The tuple $(S_1, S_2, D,
p_1, p_2, ..., p_{16})$ has the following properties:
\begin{itemize}
\item $S_1\cup S_2$ is projective and polarized with ample divisor
  $C+kF$; this is true if and only if
\begin{equation}\label{E325}
\CO_{S_1}(C + k F)\bigg |_D = \CO_{S_2}(C + kF) \bigg |_D
\end{equation}
where $C$ and $F$ are the generators of $\Pic(S_i)$ given as before;
\item the $16$ rational double points $p_1, p_2, ..., p_{16}$
  satisfies
\begin{equation}\label{E327}
\CO_{D}(p_1 + p_2 + ... + p_{16}) = \CO_D(- K_{S_1} - K_{S_2}). 
\end{equation}
\end{itemize}
The above properties are shared by all $S_1\cup S_2$'s as
degenerations of general $K3$ surfaces. In this case, the general fibers
of $\X$ are special $K3$ surfaces. So $S_1\cup S_2$ has the following
additional properties:
\begin{itemize}
\item $\Pic(S_1\cup S_2)$ has rank $2$ and is generated by $C$ and
  $F$, i.e.,
\begin{equation}\label{E328}
\CO_{S_1}(C)\bigg |_D = \CO_{S_2}(C) \bigg |_D
\ \text{and}\
\CO_{S_1}(F)\bigg |_D = \CO_{S_2}(F) \bigg |_D;
\end{equation}
\item two of the $16$ points $p_1, p_2, ..., p_{16}$, say $p_1$ and $p_2$,
  satisfy
\begin{equation}\label{E329}
\CO_D(p_1 + p_2) = \CO_D(F).
\end{equation}
\end{itemize}
It is clear that the divisor $C$ on $S_1\cup S_2$ deforms to $G$ on
the general fibers. By \eqref{E329}, there is a unique curve $J_i\in
|F|$ on $S_i$ that passes through $p_1$ and $p_2$ for $i=1,2$. Using
the deformational arguments in \cite{C1}, we see that $J_i$ deforms to
$E_i$ on the general fibers.

Let $M = M_1 \cup M_2$ be a curve in the pencil $|F|$ that passes
through one of the $16$ double points other than $p_1$ and $p_2$, say
$p_3\in M_i$ for $i=1,2$. Again using the arguments in \cite{C1}, $M$
deforms to a nodal rational curve in $|E_1+E_2|$ on a general fiber.

Next, let $N = N_1\cup N_2$ be a curve in the linear series $|C +
(k-1)F|$, where $N_i\subset S_i$ is an irreducible curve in $|C +
(k-1)F|$ that meets $D$ only at one point. By \cite{C1}, $N$ can be
deformed to a nodal rational curve on a general fiber.

Since $p_3$ is a general point on $D$, it is easy to see that $M$ and
$N$ meet transversely. The same is true for a deformation of $M$ and
$N$. We are done.
\end{proof}

Let $M\cap N = \{r_1, r_2\}$ and $E_i\cap N = q_i$ for
$i=1,2$. Obviously, 
\begin{equation}\label{E330}
\gamma = E_1 - E_2\in M^\perp \cap N^\perp
\end{equation}
and
$h(\gamma)$ is exactly the cross ratio of the four points $r_1, r_2,
q_1, q_2$ lying on $N^\nu\isom \P^1$. So \conjref{CONJ001} holds for $(X, M, N)$ if we
can show that the moduli of $(r_1, r_2, q_1, q_2)$, as four points on
$\P^1$, varies as $(X, M, N)$ deforms.

\begin{prop}\label{PROP005}
\conjref{CONJ001} holds for $(X, M, N)$ constructed above.
\end{prop}

\begin{proof}
We use the same degeneration of $(X, M, N)$ as in the proof of
\propref{PROP004}. Let $\X$ be the family of surfaces constructed
there. After a base change if necessary, we have two families of rational curves
$\M$ and $\N\subset \X$ over $\Gamma$ whose central fibers $\M_0 = M$
and $\N_0 = N$ are the curves constructed there. Also we have two
families of rational curves $\E_1$ and $\E_2\subset\X$ over $\Gamma$
with $\E_i \cap \X_0 =
J_i$. And we have two distinct sections $R_1$ and $R_2$ of $\X/\Gamma$ with
$R_i\subset \M\cap \N$. Let $Q_i = \E_i \cap \N$ for $i=1,2$. Now we
have a map
\begin{equation}\label{E331}
\lambda: \Gamma \to \overline{M}_{0,4}\isom \P^1
\end{equation}
sending
\begin{equation}\label{E332}
t\in \Gamma\xrightarrow{\lambda} (\N_t^\nu, R_1\cap \N_t, R_2\cap \N_t,
Q_1\cap \N_t, Q_2\cap \N_t)
\end{equation}
where $M_{0,4}$ is the moduli space of $\P^1$ with four marked points
and $\overline{M}_{0,4}$ is its stable closure. It suffices
to show that $\lambda$ is dominant. This is more or less obvious from
the construction of $N$. Since $N = N_1\cup N_2$ and the four points
$R_i\cap \X_0$ and $Q_i\cap \X_0$ have two on each component $N_i$,
$\lambda(0)$ must belong to $\overline{M}_{0,4}\backslash M_{0,4}$. On
the other hand, $\lambda(t) \in M_{0,4}$ for $t$ general. So $\lambda$
is nonconstant and hence dominant.
\end{proof}

Finally, we need to verify the following.

\begin{prop}\label{PROP006}
Let $W$ be a family of polarized $K3$ surfaces over the disk $\Gamma\isom \{|t| < 1\}$
whose general fibers are $K3$ surfaces with Picard rank $1$ and 
whose central fiber $W_0 = X$ is a $K3$ surface with Picard lattice
\eqref{E324}. Suppose that the Kodaira-Spencer class associated to $W$ is nonzero.
Let $\C$ and $\D\subset W$ be two distinct families of
rational curves over $\Gamma$ with $\C_0 = \D_0 = M\cup N$, where $M$
and $N$ are rational curves on $X$ given as above. Then there exists a
section $P\subset W$ of $W/\Gamma$
such that $P\subset \C\cap \D$ and $P_0 \in M\backslash N$.
\end{prop}

\begin{proof}
By \propref{PROP004}, $M$ and $N$ meet transversely at two points
$r_1$ and $r_2$. Note that $M$ is a rational curve with one node
$p$. Clearly, $p\ne r_i$ for $i=1,2$. Actually, we will show that $P$
can be chosen such that $P_0 = p$.

The two families $\C$ and $\D$ are two deformations of the union $M\cup N$, each smoothing out one of $r_i$. WLOG,
suppose that $\C$ smooths out $r_2$ and $\D$ smooths out $r_1$. By that we mean there are two sections $R_1$ and
$R_2$ of $W/\Gamma$ such that $R_i\cap W_0 = r_i$ and $R_i\cap W_t$ is a node of $\C_t$ if $i = 2$ and a node of
$\D_t$ if $i = 1$.

We blow up $W$ along $M$. The same technique was used in \cite{C2} and \cite{C-L2}.
Let $\pi: \wt{W}\to W$ be the blowup with exceptional divisor $Q$.
Let $\wt{\C}, \wt{\D}, \wt{X}$ and $\wt{N}$
be the proper transforms of $\C, \D, X$ and $N$ under $\pi$, 
respectively.
The central fiber $\wt{W}_0$ is the union of two surfaces $\wt{X}$ and $Q$
which meet along a curve $\wt{M}\isom M$. Clearly, $\wt{N}\subset \wt{X}$ meets $\wt{M}$
transversely at two points $\wt{r}_1$ and $\wt{r}_2$ over $r_1$ and $r_2$. By \cite{C2},
we see that $\pi: Q\to M$ is a $\P^1$ bundle and the normalization $Q^\nu$ of $Q$ is
$\P^1\times \P^1$.
And the threefold $\wt{W}$ has a rational double
point $q\in Q\backslash \wt{M}$ over $p\in M$. Using the techniques in \cite{C2}, we see that
the curve $C = \wt{\C}\cap Q$ is a section of $Q/M$ satisfying
\begin{itemize}
\item $C$ meets $\wt{M}$ transversely at $\wt{r}_1$;
\item $C\cap \pi^{-1}(p) = q$.
\end{itemize}
Clearly, these two conditions determines $C$ uniquely. Similarly, $D = \wt{\D}\cap Q$ is a section of $Q/M$ satisfying
\begin{itemize}
\item $D$ meets $\wt{M}$ transversely at $\wt{r}_2$;
\item $D\cap \pi^{-1}(p) = q$.
\end{itemize}
Therefore, $C\ne D$ and hence $\wt{C}_0$ and $\wt{D}_0$ meet properly at $q$. Consequently, there is a section $\wt{P}$ of $\wt{W}/\Gamma$ and hence a section $P$ of $W/\Gamma$ such that 
$P\subset \C\cap \D$ and $P(0) = p$. Indeed, there are exactly two such sections since the preimages of $q$
under the normalization $Q^\nu\to Q$ consist of two points.
\end{proof}

Now we can conclude that for a general polarized $K3$ surface $(X, L)$, the
image $\IM(\underline{r}_{3,1})\tensor \BR$ is nontrivial.

\section{Appendix: A generalized archimedean pairing}\label{SEC003}

In this section, and for each $m\geq 0$, we construct a pairing
on the cycle level, involving the equivalence relation
in the definition of Bloch's higher Chow groups $\CH^r(X,m)$ defined below.
The case when $m=0$ has already been defined in \ref{SS002}, and
the nature of this pairing is more akin to the archimedean height
pairing defined in the literature.
Although we have only used this pairing in the special instance
when $m=0$, a general construction of this pairing for all
$m$ is in order. We first recall that two subvarieties $V_1,\ V_2$
of a given variety  intersect properly if codim$\{V_1\cap V_2\} \geq$ codim $V_1 \ +$
codim $V_2$. This notion naturally extends to algebraic cycles.

\medskip
\noindent
(i) {\it Higher Chow groups.}\ Let $W/{\BC}$ a 
quasi-projective variety. Put
$z^r(W) =$ free abelian group generated by subvarieties of
codimension $r$ in $W$. Consider the $m$-simplex:
\[
\Delta^m = {\rm Spec}\biggl\{\frac{{\BC}[t_0,\ldots,t_m]}
{\big(1-\sum_{j=0}^mt_j\big)}\biggr\} \simeq {\BC}^m.
\]
We set 
\[
z^r(W,m) = \biggl\{\xi\in z^k(W\times \Delta^m)\ \biggl|
\]
\[ \xi\ 
\text{\rm meets\ all\ faces}\ \{t_{i_1} =\cdots= t_{i_\ell} = 0,\ \ell \geq 1\}
\ \text{\rm properly}\biggr\}.
\]
Note that $z^r(W,0) = z^r(W)$. Now set
$\partial_j : z^r(W,m) \to z^r(W,m-1)$, the restriction to $j$-th face
given by $t_j=0$. The boundary map
$\partial = \sum_{j=0}^m(-1)^j\partial_j : z^k(W,m)\to z^k(W,m-1)$, 
satisfies $\partial^2 = 0$.

\begin{defn}\label{DEFNA1} (\cite{B})
${\CH}^{\bullet}(W,\bullet) =$ homology of $\big\{z^{\bullet}
(W,\bullet),\partial\big\}$.
 We put ${\rm CH}^k(W) := {\rm CH}^k(W,0)$. 
 \end{defn}
 
\noindent
(ii) {\it Cubical version.}\ Let $\square^m:= ( \P^1 \backslash
\{1\})^m $ with coordinates $z_i$ and $2^{m}$ codimension one faces
obtained by setting $z_i=0,\infty$, and boundary maps
$\partial = \sum (-1)^{i-1}(\partial_{i}^{0}-\partial_{i}^{\infty})$,
where $\partial_{i}^{0},\ \partial_{i}^{\infty}$ denote
the restriction maps to the faces $z_{i}=0,\ z_{i}=\infty$ 
respectively. The rest of the definition is completely
analogous for $z^r(X,m) \subset z^r(X\times \square^m)$,
except that one has to quotient out by the
subgroup $z^r_{\dgt}(X,m) \subset z^r(X,m)$ of degenerate cycles
obtained via pullbacks $\pr_j^{\ast} : z^r(X,m-1) \to z^r(X,m)$,
$\pr_j : X\times \square^m \to X\times \square^{m-1}$ the
$j$-th canonical projection. 
It is known that both complexes are quasi-isomorphic.

In this section we will adopt the cubical version of $\CH^{\bullet}(W,\bullet)$.
The intersection product for cycles in the cubical version, is easy
to define. On the level of cycles, and in $W\times W\times \square^{m+n}$,
one has
\[ 
z^r(W,m) \times z^k(W,n) \to z^{r+k}(W\times W,m+n);
\]
however the pullback along the diagonal
\[
z^{r+k}(W\times W,m+n) \to  z^{r+k}(W,m+n),
\]
is not well-defined, even for smooth $W$. In
particular, for smooth $W$, the issue of
when an intersection product is
defined, which is a general position statement
involving proper intersections, has to be addressed since we will
be working on the level of cycles.
On the level of
Chow groups,  Bloch's \lemref{LEMA1} below (adapted to the cubical
situation) guarantees a pullback for smooth $W$:
\[
\CH^{\bullet}(W\times W,\bullet) \to \CH^{\bullet}(W,\bullet),
\]
and hence an intersection product for smooth $W$. 
Let $X$ be a projective algebraic manifold of dimension $d$, and let
$z^r_{\rat}(X,m) := \partial\big(z^r(X,m+1)\big)\subset z^r(X,m)$ be the 
equivalence relation subgroup defining the higher
Chow groups $\CH^r(X,m)$. Now introduce
\[
\Xi^0(r,m,X) = 
\]
\[
\big\{(\xi_1,\xi_2)\in z^r_{\rat}(X,m) \times z^{d-r+m+1}_{\rat}(X,m)\ \big|\
|\xi_1|\cap |\xi_2| = \emptyset\big\},
\]
\[
\Xi^+(r,m,X) = \biggl\{ (\xi_1,\xi_2) \in \Xi^0(r,m,X) \ \biggl|\
\begin{matrix} \xi_1 = \partial\xi_1', \ \text{\rm where}\\
 \xi_1'\cap \xi_2\ \text{\rm is\ defined}\\
\text{\rm in}\ z^{d+m+1}(X,2m+1)\end{matrix}\biggr\},
\]
\[
\Xi^-(r,m,X) = \biggl\{ (\xi_1,\xi_2) \in \Xi^0(r,m,X) \ \biggl|\
\begin{matrix} \xi_2 = \partial\xi_2', \ \text{\rm where}\\
 \xi_1\cap \xi_2'\ \text{\rm is\ defined}\\
\text{\rm in}\ z^{d+m+1}(X,2m+1)\end{matrix}\biggr\},
\]
\[
\Xi(r,m,X) = \Xi^+(r,m,X) \bigcap \Xi^-(r,m,X).
\]
Let $\BR(p) = \BR (2\pi\sq)^p$. Note that $\BC = \BR(p) \oplus \BR(p-1)$,
and hence there is a projection $\pi_p : \BC \twoheadrightarrow  \BR(p)$.

\begin{thm}\label{THMA1}
There are natural pairings
\[
\langle \ ,\ \rangle^+_m :  \Xi^+(r,m,X) \to \BR(m),
\]
\[
\langle \ ,\ \rangle^-_m :  \Xi^-(r,m,X) \to \BR(m),
\]
which satisfy the following:
\medskip

{\rm (i)} On $\Xi(r,m,X)$, $\langle\  ,\ \rangle^+_m = (-1)^m\langle\ ,\ \rangle_m^-$.

\medskip

{\rm (ii)} (Bilinearity) If $(\xi_1^{(1)},\xi_2),\ (\xi_1^{(2)},\xi_2)
 \in \Xi^+(r,m,X)$, then
\[
\langle \xi_1^{(1)}+\xi_1^{(2)},\xi_2\rangle^+_m = 
\langle \xi_1^{(1)},\xi_2\rangle^+_m +
\langle \xi_1^{(2)},\xi_2\rangle^+_m.
\]
If $(\xi_1,\xi_2^{(1)}),\ (\xi_1,\xi_2^{(2)})\in \Xi^-(r,m,X)$, then
\[
\langle\xi_1,\xi_2^{(1)}+\xi_2^{(2)}\rangle^-_m = 
\langle \xi_1^{(1)},\xi_2\rangle^-_m +
\langle \xi_1^{(2)},\xi_2\rangle^-_m.
\]

{\rm (iii)} (Projection formula)  Let $\pi: X\to Y$ be a flat surjective morphism between two smooth
projective varieties, with $\dim X=d$. Then
$\langle \xi_1, \pi^* \xi_2\rangle_m^{\pm} = \langle \pi_* \xi_1,
\xi_2\rangle_m^{\pm}$ for all $\xi_1\in z_{\rat}^r(X,m)$ and
$\xi_2 \in z_{\rat}^{d-r+m+1}(Y,m)$ with $(\pi_* \xi_1,
\xi_2)\in \Xi^{\pm}(r+s-d,m,Y)$, where $s := \dim Y$.

\medskip

{\rm (iv)} (Reciprocity)  $\langle \xi_1,\xi_2\rangle_m = (-1)^m\langle \xi_2 ,\xi_1 \rangle_m$,
where
\[
\langle\ ,\ \rangle_m := \langle\ ,\ \rangle_m^+\big|_{\Xi(r,m,X)} = 
(-1)^m\langle\ ,\ \rangle_m^-\big|_{\Xi(r,m,X)}.
\]
\end{thm}

\begin{proof} We first recall the definition of real Deligne
cohomology. Let $\D_X^{\bullet}$ be the sheaf of 
complex-valued currents acting on $C^{\infty}$ complex-valued
compactly supported $(2d-\bullet)$-forms, where we recall $\dim X=d$. One has 
a decomposition into Hodge type:
\[
\D_X^{\bullet} = \bigoplus_{p+q=\bullet}\D_X^{p,q},
\]
where $\D_X^{p,q}$ acts on $(d-p,d-q)$ forms, with Hodge
filtration,
\[
F^r\D_X^{\bullet} = \bigoplus_{p+q=\bullet,p\geq r}\D_X^{p,q}.
\]
One has filtered quasi-isomorphism of complexes,
\[
(F^r)\Omega_X^{\bullet} \hookrightarrow  (F^r)\E_X^{\bullet} \hookrightarrow
(F^r)\D_X^{\bullet},
\]
where $\E_X^{\bullet}$ (resp. $\Omega_X^{\bullet}$) is the sheaf complex of germs
of complex-valued $C^{\infty}$ (resp. holomorphic) forms on $X$. Let us put
$\E_{X,\BR}^{\bullet} := $ sheaf complex of germs of real $C^{\infty}$ forms,
and likewise $\D^{\bullet}_{X,\BR}$ the sheaf complex of real-valued currents.
We define $\D_{X,\BR(p)}^{\bullet} = \D_{X,\BR}^{\bullet}\otimes_{\BR}\BR(p)$,
$\E_{X,\BR(p)}^{\bullet} = \E_{X,\BR}^{\bullet}\otimes_{\BR}\BR(p)$. The
global sections of a given sheaf ${\mathcal S}$ over $X$ will be
denoted by ${\mathcal S}(X)$.  Next, for a morphism
of complexes $\lambda : A^{\bullet} \to C^{\bullet}$, we recall the cone complex:
\[
\text{\rm Cone}(A^{\bullet} {\buildrel \lambda\over\longrightarrow} B^{\bullet}) = A^{\bullet}[1]
\oplus B^{\bullet},
\]
with differential
\[
\delta_D : A^{q+1}\oplus B^q \to A^{q+2}\oplus B^{q+1},
\quad (a,b) \ {\buildrel \delta_D\over \mapsto} \ (-da,\lambda(a) + db).
\]

\begin{defn}\label{DEFN2} The real Deligne chomology of $X$ is
given by
\[
H^i_{\D}(X,\BR(j)) := H^i\big(\text{\rm Cone}\big(F^j\D_X^{\bullet}(X) \ {\buildrel
-\pi_{j-1}\over\longrightarrow}\ \D^{\bullet}_{X,\BR(j-1))}(X)\big)[-1]\big).
\]
\end{defn}
We now recall the description of the real regulator
\[
R_{r,m,X} : \CH^r(X,m) \to H_{\D}^{2r-m}(X,\BR(r)),
\]
(see \cite{Go}, as well as \cite{K}, \cite{KLM}).
 For nonzero rational functions $\{f_1,...,f_m\}$
 defined on a complex variety, we introduce the real current $2\pi\sq R_m$, where
 (\cite{Go}):
 \[
 R_m(f_1,...,f_m) := 
 \]
 \[
( 2\pi\sq)^{-m}\Alt_m\sum_{j\geq 0}\frac{1}{(2j+1)!(m-2j-1)!}\log|f_1|
d\log|f_2|\wedge
 \]
 \[
 \cdots \wedge d\log|f_{2j+1}|\wedge d\sq\arg(f_{2j+2})\wedge\cdots\wedge
 d\sq\arg(f_m),
 \]
 and where 
 \[
 \Alt_mF(x_1,...,x_m) := \sum_{\sigma\in S_m}(-1)^{|\sigma|}F(x_{\sigma(1)},...,
 x_{\sigma(m)}),
 \]
 is the alternating operation, and $S_m$ is the group of permutations on
 $m$ letters. Consider $\square^m$ with affine coordinates $(z_1,...,z_m)$,
 and introduce the operators
 \[
 R_{\square} := R_m(z_1,...,z_m),\quad \Omega_{\square} = \bigwedge_{j=1}^md\log z_j,
 \]
For a $\xi\in z^r(X,m)$ we consider the currents on $X$:
\[
R_m(\xi) := \int_{\xi}\pr_{\square}^{\ast}(R_{\square})\wedge \pr_X^{\ast}(-),
\quad \Omega_m(\xi)  := \int_{\xi}\pr_{\square}^{\ast}(\Omega_{\square})\wedge \pr_X^{\ast}(-).
\]
It is easy to check that 
\[
\big(R_m(\xi), \Omega_m(\xi)\big) = (0,0) \ \text{\rm for}\ \xi\in z_{\dgt}^r(X,m),
\]
and that $(2\pi\sq)^m d R_{\square} = \pi_{m-1}(\Omega_{\square})$ as
forms (or as currents acting on forms compactly supported away from the $2^m$ faces
of $\square^m$).
Then up to a twist, $R_{r,m,X}$ is induced by:
\[
\xi\in z^r(X,m) \mapsto \big(\Omega_m(\xi),(2\pi\sq)^m R_m(\xi)\big),
\]
(see \cite{K} or \cite{KLM}, where it follows that 
$(2\pi\sq)^m d R_{m}(\xi) = \pi_{m-1}(\Omega_{m}(\xi))$ if $\partial\xi=0$).
For $m=0$, note that $ (\Omega_0(\xi),(2\pi\sq)^m R_0(\xi)) = (1_{\xi},0)$, where $1_{\xi}$ defines
the current on $X$ given by integration over $\xi$.

\medskip

Now to the proof of the theorem. For simplicity, we will assume given
$(\xi_1,\xi_2)\in \Xi(r,m,X)$. By definition,
this  implies that 
 \[
 \xi_1'\cap\xi_2,\
 \xi_1\cap\xi_2' \in z^{d+m+1}(X,2m+1),
 \]
 which is important in ensuring that the integrals given below converge. 
 The prescription for the following pairings is based on the formalism
 of a cup product operation on Deligne complexes.
 Namely we define:
 \[
 \langle\xi_1,\xi_2\rangle^+_m := (2\pi\sq)\biggl[\int_{\xi_1'\cap\xi_2}R_{m+1}(\xi_1')\wedge \pi_m(\Omega_m
 (\xi_2)) \ +
 \]
 \[
  (-1)^{m+1}\int_{\xi_2\cap \xi_1'}\pi_{m+1}(\Omega_{m+1}
  (\xi_1'))\wedge R_m(\xi_2) \biggr],
 \]
 \[
 = (2\pi\sq)\biggl[\int_{\xi_1'\cap\xi_2}R_{m+1}(\xi_1')\wedge \pi_m(\Omega_m
 (\xi_2))\biggr],
 \]
 using the fact that $\dim |\xi_1\cap \xi_2'| \leq m$ and that $\Omega_{m+1}(\xi_1')$
 is a holomorphic current involving $m+1$ holomorphic differentials. Likewise,
  \[
 \langle\xi_1,\xi_2\rangle^-_m := (2\pi\sq)\biggl[\int_{\xi_1\cap\xi_2'}R_{m}(\xi_1)\wedge \pi_m(\Omega_{m+1} (\xi_2')) \ +
 \]
 \[
  (-1)^{m}\int_{\xi_2'\cap \xi_1}\pi_{m}(\Omega_{m}
  (\xi_1))\wedge R_{m+1}(\xi_2') \biggr],
 \]
 \[
 = (2\pi\sq)(-1)^m\biggl[\int_{\xi_2'\cap \xi_1}\pi_{m}(\Omega_{m}
  (\xi_1))\wedge R_{m+1}(\xi_2') \biggr].
 \]
  The interpretation of these integrals is
 fairly clear.  For instance
 \[
\int_{\xi_1'\cap\xi_2}R_{m+1}(\xi_1')\wedge \pi_m(\Omega_m
 (\xi_2))
 \]
 means the following: In the product
 \[
 X\times X\times \square^{m+1}\times\square^m,
 \]
 let
 \[ 
 \delta : X\times  \square^{m+1}\times\square^m \hookrightarrow 
  X\times X\times \square^{m+1}\times\square^m,
  \]
 be induced from the diagonal embedding $X\hookrightarrow X\times X$,
 together with the identity map on $\square^{m+1}\times\square^m$,
 and $\pr_j$, $j=1,2,3,4,$  the canonical
 projections.
 Then
  \[
\int_{\xi_1'\cap\xi_2}R_{m+1}(\xi_1')\wedge \pi_m(\Omega_m
 (\xi_2)) = \int_{\xi_1'\cap\xi_2}\delta^{\ast}\biggl[
 \pr_3^{\ast}\big(R_{\square^{m+1}}\big)\wedge \pr_4^{\ast}\big( \pi_m(\Omega_{\square^m})
 \big)\biggr].
 \]

  The relationship between $\langle\ ,\  \rangle_m^+$ and
$\langle\ ,\  \rangle_m^-$ is precisely where reciprocity comes into play.
 Observe that 
 \[
 \partial (\xi_1'\cap \xi_2') = \xi_1\cap \xi_2'  +  (-1)^{m+1} \xi_1'\cap \xi_2,
 \]
 and by \cite{B},
 \[
 \xi_1\cap \xi_2'  = (-1)^{m(m+1)} \xi_2' \cap \xi_1 = \xi_2' \cap \xi_1.
 \]
Next,
 \[
  \partial (\xi_1'\cap \xi_2') \mapsto 0\in H_{\mathcal D}^{2d+1}(X,\BR(d+m+1)),
  \]
 and together with the fact that max$\{\dim |\xi_1'\cap \xi_2|, \dim |\xi_1\cap \xi_2'|\} \leq m$,
 and that $ \Omega_{m+1}(\xi_1')\cup  \Omega_{m}(\xi_2)$,  $\Omega_m(\xi_1)
 \cup \Omega_{m+1}(\xi_2')$ are currents involving $2m+1$ holomorphic differentials,
 hence
 \[
 \Omega_{m+1}(\xi_1')\cup  \Omega_{m}(\xi_2)= 0 = \Omega_m(\xi_1)
 \cup \Omega_{m+1}(\xi_2'),
 \]
  we arrive at
 \[
(-1)^m \langle\xi_2,\xi_1\rangle^+_m := (2\pi\sq)\biggl[\int_{\xi_2'\cap\xi_1}R_{m+1}(\xi_2')\wedge \pi_m(\Omega_m(\xi_1)) \biggr]
 \]
 \[ 
 = \langle \xi_1,\xi_2\rangle_m^-.
 \]
 I.e.
 \[
 \langle \xi_1,\xi_2\rangle_m^+ = (-1)^m \langle \xi_2,\xi_1\rangle^+_m = 
 (-1)^m \langle \xi_1,\xi_2\rangle^-_m.
 \]
 The formula for $\langle \xi_1,\xi_2\rangle^{\pm}_m$ above
 essentially involves the product structure for a complex of forms
 defining real Deligne cohomology (see \cite{EV}), translated in the
 terminology of our special class of currents above. Similar ideas
 have been worked out in \cite{L1}. If either $\partial\xi_1' = 0$ or $\partial \xi_2' = 0$,
 then this reduces to a cup product in Deligne cohomology of the Beilinson (real) regulator of 
 a higher Chow cycle, together with one which is nullhomologous (in Deligne cohomology),
 which is zero.
 Hence this pairing does not depend on the choices
 of $\xi_j'$.    The theorem essentially follows from this.
 \end{proof}
 
Now recall that
 \[
\langle\ ,\ \rangle_m := \langle\ ,\ \rangle_m^+\big|_{\Xi(r,m,X)} = 
(-1)^m\langle\ ,\ \rangle_m^-\big|_{\Xi(r,m,X)}.
\]

It is natural to pose the following nondegeneracy type question.

\begin{question}
 Suppose that 
$\langle\xi_1 ,\xi_2 \rangle_m = 0$ for all $\xi_2$ with
$(\xi_1,\xi_2)\in \Xi(r,m,X)$. Is it the case that
$\xi_1\in z_{\dgt}^r(X,m)$? Or if that is too strong,
a possibly weaker statement
could be $\Omega_m(\xi_1) = 0$ (as a current on $X$)?
\end{question}

In the case $m=0$, \propref{PROP002} answers the above question 
definitively in the affirmative.
For our next result, and as a partial answer to the above question,
the reader can consult \cite{El-V}  (p. 187) for the business of
a graph map.

\begin{prop}\label{PROP401}
Let  $\xi\in z^r(X,m)$ be represented as the
graph of a cycle of the form:
\[
\xi_1 = \sum_{\alpha\in I}\big(\{f_{1,\alpha},...,f_{m,\alpha}\},D_{\alpha}\big),
\]
where  $\{f_{1,\alpha},...,f_{m,\alpha}\}\in \BC(D_{\alpha})^{\ast}$, and
where $\{D_{\alpha},\ \alpha\in I\}$ are distinct irreducible 
subvarieties with codim$_XD_{\alpha} = r-m$. Then
\[
\langle\xi_1 ,\xi_2 \rangle_m = 0\ {for\ all}\  \xi_2\  {with}\
(\xi_1,\xi_2)\in \Xi(r,m,X)\ \Rightarrow \Omega_m(\xi_1) = 0.
\]
\end{prop}

\begin{proof} Let us assume that  $\Omega_m(\xi_1) \ne 0$. This
is equivalent to saying
$\pi_m(\Omega_m(\xi_1)) \ne 0$.
Then for some $0\in I$, 
\[
\int_{D_0}\pi_m\biggl[ \bigwedge_1^m d\log f_{0,j}\biggr]\wedge (-)  \ne 0.
\]
Now let $E\subset X$ be a general choice of irreducible subvariety
with codim$_XE = d-r$, and let $\{g_1,...,g_{m+1}\}\in \BC(E)^{\ast}$. By distinctness
of the $\{D_{\alpha},\ \alpha\in I\}$, we can choose $g_1$ such that $g_1\big|_{D_{\alpha}} \equiv 1$
for all $\alpha\ne 0$. If we put $\xi_2'$ to correspond to the graph of
$\big(\{g_1,...,g_{m+1},E\big)$, then using the dictionary
\[
\big(\{f_{0,1},...,f_{0,m},D_0\big) \leftrightarrow \big(\{f_1,...,f_m\},D\big),
\]
we have

\begin{equation}\label{E100}
(2\pi\sq)^{-1}\langle \xi_1,\xi_2\rangle_m = \pm\int_{D\cap E}\pi_m(\Omega_m(f_1,...,f_m))\wedge
R_{m+1}(g_1,...,g_{m+1}).
\end{equation}

Next, by replacing $g_1$ by 
\[
g_{1,\varepsilon} := 1 \ +\  (h-1) \biggl(\frac{(g_1-1)}{(g_1-1) + \varepsilon}\biggr),
\]
where $h$ is {\it any} rational function and $\varepsilon > 0$, and letting $\varepsilon \mapsto 0^+$, it follows that we can reduce to the situation of an {\it arbitrary} choice
of $\{g_1,...,g_{m+1}\}$ in equation (\ref{E100}).  In particular, if we consider
\[
F := (f_1,...,f_m) : D \cap E \twoheadrightarrow \big(\P^1\big)^{\times m},
\]
then 
\[
\pi_m(\Omega_m(f_1,...,f_m) = F^{\ast}\big(\pi_m(\Omega_m(z_1,...,z_m)\big).
\]
Thus  equation (\ref{E100}) and the non-triviality of $\langle \xi_1,\xi_2\rangle_m$
reduces to showing the non-triviality of
\[
\int_{\BC^m}\pi_m(\Omega_m(z_1,...,z_m))\wedge
R_{m+1}(w_1,...,w_{m+1}),
\]
for a choice of functions $\{w_1,...,w_{m+1}\}$ of $z := (z_1,...,z_m)$. That
one can find such functions is not difficult, and left
to the reader.
\end{proof}

 It is instructive to work out a couple example cases.
 
 \bigskip
 \noindent
 \underbar{Case $m=0$}: 
 We recover the pairing in \ref{SS002}.
 We first recall the following key result.

\begin{lem}\label{LEMA1} (\cite{B}, Lemma (4.2)) Let $Y$ be a smooth,
quasi-projective $k$-variety and let $y = \{Y_i\}$ be
a finite collection of closed subvarieties. Then the inclusion
$z^{\bullet}_y(Y,\bullet) \subset z^{\bullet}(Y,\bullet)$ is a quasi-isomorphism.
\end{lem}
\noindent
Here $z^{\bullet}_y(Y,\bullet) \subset z^{\bullet}(Y,\bullet)$ is
the subcomplex of cycles that meet $y\times \square^{\bullet}$
properly.
In the case $m=0$, we have $\Xi^0(r,0,X) \subset z^r_{\rat}(X)\times
z^{d-r+1}_{\rat}(X)$. Let $(\xi_1,\xi_2)\in  z^r_{\rat}(X)\times
z^{d-r+1}_{\rat}(X)$. By considering the cases where $y = |\xi_j|$, $j=1,2$,
it follows that
\[
\Xi^0(r,0,X) = \Xi^+(r,0,X) = \Xi^-(r,0,X) = \Xi(r,0,X),
\]
and thus we have a pairing
\[
\langle \ ,\ \rangle := \langle\ ,\ \rangle_0 : z^r_{\rat}(X)\times z_{\rat}^{d-r+1}(X) \to
\BR,
\]
defined for all pairs $(\xi_1,\xi_2)$ where $|\xi_1|\cap |\xi_2| = \emptyset$.
Let $\xi_1 := \DIV(f,D) \in z^r_{\rat}(X,0)$,
 $\xi_2 := \DIV(g,E)\in z_{\rat}^{d-r+1}(X,0)$ be given. In this
 case $D$ and $E$ are irreducible subvarieties of $X$
 of codim$_XD = r-1$ and codim$_XE = d-r$, and $f\in \BC(D)^{\ast}$,
 $g\in \BC(E)^{\ast}$. Then
 \[
 \langle\xi_1,\xi_2\rangle_0 := \int_{D\cap \xi_2}\log|f| =  \int_{E\cap \xi_1}\log|g|
 =:   \langle \xi_2,\xi_1\rangle_0.
 \]
The reader can check that \propref{PROP002} can be deduced from
\propref{PROP401}. 

\bigskip
 \noindent
 \underbar{Case $m=1$}: Let 
 \[
 \xi_1 := T(\{f_1,f_2\},D) \in z_{\rat}^r(X,1),
 \]
 and  
 \[
 \xi_2 := T(\{g_1,g_2\},E)\in z_{\rat}^{d-r+2}(X,1)
 \]
  be given, where
 $T$ is the Tame symbol. In this
 case $D$ and $E$ are irreducible subvarieties of $X$
 of codim$_XD = r-2$ and codim$_XE = d-r$, and $f_j\in \BC(D)^{\ast}$,
 $g_j\in \BC(E)^{\ast}$, $j=1,2$. Note that
 \[
 \xi_1 :=
 \sum_{\cd_DD_i=1}\biggl((-1)^{\nu_{D_i}(f_1)\nu_{D_i}(f_2)}
 \frac{f_1^{\nu_{D_i}(f_2)}}{f_2^{\nu_{D_i}(f_1)}},D_i\biggr),
 \]
  \[
 \xi_2 :=
 \sum_{\cd_EE_j=1}\biggl((-1)^{\nu_{E_j}(g_1)\nu_{E_j}(g_2)}
 \frac{g_1^{\nu_{E_j}(g_2)}}{g_2^{\nu_{E_j}(g_1)}},E_j\biggr).
 \]
 Then 
 $2\pi  \langle \xi_1,\xi_2\rangle_1 =$
 \[
\sum_j
\int_{D\cap E_j}\big[\log|f_1|d\arg(f_2) - \log|f_2|d\arg(f_1)\big]
 \wedge \pi_1\biggl[d\log \frac{g_1^{\nu_{E_j}(g_2)}}{g_2^{\nu_{E_j}(g_1)}}
 \biggr].
  \]

\end{document}